\documentclass[a4paper,11pt]{amsart}
\usepackage{amstext,amsfonts,amsthm,graphicx,amssymb,amscd,epsfig}
\usepackage{amsmath}
\usepackage[ansinew]{inputenc}    %   Windows Latin 1
\usepackage{psfrag}
\usepackage{mathrsfs}
\usepackage{xypic}
\usepackage{graphicx}
\usepackage[all,knot,arc,import,poly]{xy}
\usepackage{comment}
\usepackage[small]{caption} 

\theoremstyle{definition}
\newtheorem{definition}{Definition}[section]
\newtheorem{ex}[definition]{Example}
\newtheorem{rem}[definition]{Remark}

\theoremstyle{plain}
\newtheorem{prop}[definition]{Proposition}

\newtheorem{coro}[definition]{Corollary}
\newtheorem{teo}[definition]{Theorem}

\newfont{\bbb}{msbm10 scaled\magstephalf}     %% R, C, Z -- 11pt

\def\reg{\operatorname{reg}}
\def\sing{\operatorname{sing}}

\title{Curvature loci of 3-manifolds}

\author{P. Benedini Riul, R. Oset Sinha AND M. A. S. Ruas}

\date{}

\address{Universidade Federal de S\~ao Jo\~ao del Rei - UFSJ, Campus Alto Paraopeba, Rodovia MG 443 - KM 7, CEP: 36497-899 - Ouro Branco - MG, Brazil}

\email{benedini@ufsj.edu.br}

\address{Departament de Matem\`{a}tiques,
Universitat de Val\`encia, Campus de Burjassot, 46100 Burjassot,
Spain}

\email{raul.oset@uv.es}

\address{Instituto de Ci\^encias Matem\'aticas e de Computa\c{c}\~ao - USP,
Av. Trabalhador s\~ao-carlense, 400 - Centro,
CEP: 13566-590 - S\~ao Carlos - SP, Brazil}

\email{maasruas@icmc.usp.br}
\thanks{Work of R. Oset Sinha partially supported by Grant PGC2018-094889-B-100 funded by MCIN/AEI/ 10.13039/501100011033 and by ``ERDF A way of making Europe"}
\thanks{Work of M.A.S.Ruas partially  supported by FAPESP Proc. 2019/21181-0 and CNPq Proc. 305695/2019-3}


\begin{document}

\begin{abstract}
We refine the affine classification of real nets of quadrics in order to obtain generic curvature loci of regular $3$-manifolds in $\mathbb{R}^6$ and singular corank $1$ $3$-manifolds in $\mathbb{R}^5$. For this, we characterize the type of the curvature locus by the number and type of solutions of a system of equations given by 4 ternary cubics (which is a determinantal variety in some cases). We also study how singularities of the curvature locus of a regular 3-manifold can go to infinity when the manifold is projected orthogonally in a tangent direction.
\end{abstract}

\maketitle

\section{Introduction}

The extrinsic geometry of 3-manifolds in $\mathbb R^6$ from a singularity theory viewpoint was  investigated by M.C. Romero Fuster, R. Binotto
and S. Costa in \cite{Carmen3var, Carmen3var2} and by D. Dreibelbis in \cite{dreibelbis}. In the article  \cite{Carmen3var}, the authors 
 study the behaviour of the curvature locus (also called curvature Veronese) of a 3-manifold immersed in $\mathbb R^n , n \geq 5,$ describing the possible topological types at the different points of a 3-manifold immersed in $\mathbb R^6.$  Their work was motivated by \cite{computeiros}, in which quadratically 
 parameterizable surfaces, of which the Steiner surface is an example, have been studied in terms of their feasibility as surface patches in computer aided geometric design. 
 
Motivated by the results in \cite{Carmen3var}, the first and third authors, and A. Sacramento, 
introduced  in \cite{BenediniRuasSacramento},  the invariants of the second fundamental form of corank 1  \, 3-manifolds in $\mathbb R^5,$ and studied properties of its second order geometry in connection with the geometry of  regular 3-manifolds in $\mathbb R^6.$ Notice that the image of  the projection of a 3-dimensional smooth manifold  into a hyperplane in $\mathbb R^6$ along a tangent direction at a point $p,$ gives (locally) a corank 1 3-manifold in a 5-dimensional space. The converse is also true, in other words, given a parametrized corank 1  3-manifold in $\mathbb R^5,$   defined in a sufficiently small neighbourhood of a point, there exists a 3-dimensional smooth surface in $\mathbb R^6$ projecting onto it (See  also \cite{benediniosetpm}).

The aim of this paper is to complement these previous investigations, by  presenting a more formal approach to the  affine classification of the curvature locus
of smooth and singular 3-manifolds in $\mathbb R^n,\,$ $n=6$ and $5,$ respectively.

The affine geometry of the second fundamental form of a $3$-dimensional manifold in $\mathbb R^6$ is the study of invariants
of the action of the affine group $\mathcal{G}=GL(3)\times GL(3)$ in the space of quadratic polynomial mappings $g:\mathbb{R}^3\rightarrow\mathbb{R}^3.$ Each such mapping generates a system of quadrics denominated a ``net of quadrics". One can find in the literature many texts about nets of quadrics  
(see  \cite{wall/duplesis,wall,wall1,wall2}).

We refine the affine classification of real nets of quadrics in order to obtain generic curvature loci of regular $3$-manifolds in $\mathbb{R}^6$ and singular corank $1$ $3$-manifolds in $\mathbb{R}^5$. For this, we characterize the type of the curvature locus by the number and type of solutions of a system of equations given by 4 ternary cubics (which is a determinantal variety in some cases). We also study how singularities of the curvature locus of a regular 3-manifold can go to infinity when the manifold is projected orthogonally in a tangent direction.

The paper is organized as follows. In Sections 2.1 and 2.2, we review the definitions of first and second fundamental forms and the notion of curvature loci of regular and singular 3-dimensional
manifolds in $\mathbb R^n, \, n\geq 4.$  In Section 2.3, we present the affine classification of real nets of quadrics given by S. A. Edwards and C.T.C. Wall 
in \cite{wall}. From this section on, we restrict the discussion of the paper to 3-dimensional regular and singular manifolds in $\mathbb R^n,\, n=6$ and $5,$ respectively. The aim in Section 3, is to characterize  the curvature locus, when this set is a substantial surface, that is, a surface with
Gaussian curvature non identically zero. The main result in this section, is Theorem \ref{class1}, in which we prove that the singular sets of the restriction of the second fundamental form to the unit sphere in $\mathbb R^3$ are a complete set of invariants for the substantial curvature loci. 
The orthogonal projection of a regular 3-manifold in $\mathbb R^6$ into $\mathbb R^5$ along a tangent direction is a corank 1  3-manifold; their second fundamental form are the same, however their curvature loci are not equal. We discuss in Section 4 the relation between these two invariants.
The goal in Section 5 is  to prove Theorem \ref{genlocus}, in which we  classify the generic curvature locus of  each $\mathcal G$-orbit of a quadratic map 
$f: \mathbb R^3 \to \mathbb R^3.$

%%%%%%%%%%%%%%%%%%%%%%%%%%%%%%%%%%%%%%%%%%%%%%%%%%%%%%%%%%%%%%%%%%%%%%%%%%%%%%%%%%%%%%%%%%%%%%%%%%%%%%%%%%%%%%%%%%%%%%%%%%%%%%%%%%%%%%%
%%%%%%%%%%%%%%%%%%%%%%%%%%%%%%%%%%%%%%%%%%%%%%%%%%%%%%%%%%%%%%%%%%%%%%%%%%%%%%%%%%%%%%%%%%%%%%%%%%%%%%%%%%%%%%%%%%%%%%%%%%%%%%%%%%%%%%%

\section{Preliminary results}

\subsection{Second order geometry of $3$-manifolds in $\mathbb{R}^n$}

Given a smooth $3$-dimensional manifold $M^{3}_{\reg}\subset\mathbb{R}^{N}$, $N>3$ and $f:U\rightarrow\mathbb{R}^{N}$
a local parametrisation of $M^{3}_{\reg}$ with $U\subset\mathbb{R}^{3}$ an open subset, let
$\{e_{1},\ldots,e_{N}\}$ be an orthonormal frame of $\mathbb{R}^{N}$ such that at any $u\in U$,
$\{e_{1}(u),e_2(u),e_{3}(u)\}$ is a basis for $T_{p}M^{3}_{\reg}$ and $\{e_{4}(u),\ldots,e_{N}(u)\}$ is a basis for
$N_{p}M^{3}_{\reg}$ at $p=f(u)$.

The second fundamental form of $M^{3}_{\reg}$ at a point $p$ is a symmetric bilinear map,
$II_{p}:T_{p}M^{3}_{\reg}\times T_pM^{3}_{\reg}\rightarrow N_{p}M^{3}_{\reg}$ given by
$II_p(v,w)=\pi_2(d^2f(v,w))$, where $\pi_2:T_p\mathbb R^N\rightarrow N_pM^3_{\reg}$ is the canonical projection.

Furthermore, the second fundamental form of $M^{3}_{\reg}$ at $p$ along a normal vector field $\nu$
is the bilinear map $II_{p}^{\nu}:T_pM^{3}_{\reg}\times T_pM^{3}_{\reg}\rightarrow\mathbb{R}$ given by
$II_{p}^{\nu}(v,w)=\langle \nu,II_p(v,w)\rangle$.

For singular $3$-dimensional manifolds with corank $1$ singularities, we shall need the following construction.
Let $M^{3}_{\sing}$ be a corank $1$ $3$-manifold in $\mathbb{R}^{N}$, $N>3$, and consider a point $p\in M^{3}_{\sing}$. The singular manifold $M^{3}_{\sing}$ will be taken as the image of a smooth map $g:\tilde{M}\rightarrow \mathbb{R}^{N}$, where $\tilde{M}$ is a smooth regular $3$-dimensional manifold and $q\in\tilde{M}$ is a corank $1$ point of $g$ such that $g(q)=p$. Also, consider $\phi:U\rightarrow\mathbb{R}^{3}$ a local coordinate system defined in an open neighbourhood $U$ of $q$ at $\tilde{M}$. Hence, we may consider a local parametrisation $f=g\circ\phi^{-1}$ of $M^3_{\sing}$ at $p$ (see the diagram below).

$$
\xymatrix{
\mathbb{R}^{3}\ar@/_0.7cm/[rr]^-{f} & U\subset\tilde{M}\ar[r]^-{g}\ar[l]_-{\phi} & M^{3}_{\sing}\subset\mathbb{R}^{N}
}
$$

At the singular point $p,$ the $2$-dimensional tangent space  $T_{p}M^{3}_{\sing}$, is given by $\mbox{Im}\ dg_{q}$, where $dg_{q}:T_{q}\tilde{M}\rightarrow T_{p}\mathbb{R}^{N}$ is the differential map of $g$ at $q$. Thus, the $(N-2)$-dimensional normal space of $M^{3}_{\sing}$ at $p$, $N_{p}M^{3}_{\sing}$, is the subspace orthogonal to $T_{p}M^{3}_{\sing}$ satisfying $T_{p}M^{3}_{\sing}\oplus N_{p}M^{3}_{\sing}=T_{p}\mathbb{R}^{N}$.

The first fundamental form of $M^{3}_{\sing}$ at $p$, $I:T_{q}\tilde{M}\times T_{q}\tilde{M}\rightarrow \mathbb{R}$ is given by
$$I(u,v)=\langle dg_{q}(u),dg_{q}(v)\rangle,\ \ \forall\ u,v\in T_{q}\tilde{M}.$$

The first fundamental form is not a Riemannian metric on $T_{q}\tilde{M}$, but a pseudometric instead.
Let $(x_1,x_2,x_3)$ be the Cartesian coordinate system in $\mathbb{R}^3$.
Taking the frame $\mathcal{B}=\{\partial x_1,\partial x_2,\partial x_3\}$ of $T_{q}\tilde{M}$, the coefficients of the first fundamental form of $M^3_{\sing}$ at $p$ with respect to $\phi$ are given by $E_{x_ix_j}(q)=I(\partial_{x_i},\partial_{x_j})=\langle f_{x_i},f_{x_j}\rangle(\phi(q))$, $1\leq i,j\leq 3$, where $f_{x_i}=\frac{\partial f}{\partial x_i}$.

Consider the orthogonal projection $\pi_2:T_{p}\mathbb{R}^{N}\rightarrow N_{p}M^{3}_{\sing}$. The second fundamental form of $M^{3}_{\sing}$ at $p$, $II:T_{q}\tilde{M}\times T_{q}\tilde{M}\rightarrow N_{p}M^3_{\sing}$ in the basis $\mathcal{B}$ of $T_{q}\tilde{M}$ is given by $II(\partial_{x_i},\partial_{x_j})=\pi_2(f_{x_{i}x_{j}}(\phi(q)))$, $1\leq i,j\leq 3$
and we extend it to the whole space in a unique way as a symmetric bilinear map.
Given a normal vector $\nu\in N_{p}M^{3}_{\sing}$, we define the second fundamental form along $\nu$, $II_{\nu}:T_{q}\tilde{M}\times T_{q}\tilde{M}\rightarrow\mathbb{R}$ given by $II_{\nu}(u,v)=\langle II(u,v),\nu\rangle$, for all $u,v\in T_{q}\tilde{M}$.

\subsection{The curvature loci}

Given a $k$-dimensional manifold $M^{k}\subset\mathbb{R}^{N}$, $N>k$, the curvature
locus at a point $p\in M^{k}$ is the set $\{II_p(u,u):\ u\in T_pM^k, I(u,u)^{\frac{1}{2}}=1\}\subset N_{p}M^{k}$. The curvature locus of a manifold contains all its second order geometry. Any isometric scalar invariant of the curvature locus is an isometric scalar invariant of the manifold since rotations in $T_pM^k$ leave invariant the locus and rotations in $N_pM^k$ rotate the locus. The contact geometry of the manifold is affine invariant. The curvature locus is not affine invariant, but the position with respect to the origin and, in some cases, the topological type are affine invariant.

\textbf{\underline {Regular case}:}

For a regular manifold $M^{3}_{\reg}\subset\mathbb{R}^{N}$, the curvature locus is also the image of the map
$\eta:\mathbb{S}^{2}\subset T_{p}M^{3}_{\reg}\rightarrow N_{p}M^{3}_{\reg}$, where
$\eta(u)=II_p(u,u)$ and is denoted by $\Delta_{v}$. The authors show in
\cite[p. 27]{Carmen3var} that
taking spherical coordinates in $\mathbb{S}^2\subset T_pM^{3}_{\reg}$, one can parametrise the curvature locus of $M^{3}_{\reg}$
at $p$ by
$\eta:\mathbb{S}^{2}\subset T_pM^{3}_{\reg}\rightarrow N_pM^{3}_{\reg},\ (\theta,\phi)\mapsto\eta(\theta,\phi)$,
where
$$
\begin{array}{cl}
    \eta(\theta,\phi)= & H+(1+3\cos(2\phi))B_1+\cos(2\theta)\sin^2\theta B_2 \\
     & +\sin(2\theta)\sin^2\phi B_3+cos\theta\sin(2\phi)B_4+\sin\theta\sin(2\phi)B_5
\end{array}
$$
with
$$
\begin{array}{c}
     H=\frac{1}{3}(f_{xx}+f_{yy}+f_{zz}),\ B_1=\frac{1}{12}(-f_{xx}-f_{yy}+2f_{zz}),  \\
     B_2=\frac{1}{2}(f_{xx}-f_{yy}),\ B_3=f_{xy},\ B_4=f_{xz},\ B_5=f_{yz}.
\end{array}
$$

The first normal space
is $N_p^1M^{3}_{\reg}=\langle H,B_1,B_2,B_3,B_4,B_5\rangle_{(p)}$. The affine hull of the curvature locus is denoted by $Aff_p$ and
the linear subspace of $N_p^1M^{3}_{\reg}$ parallel to $Aff_p$ by $E_p$.
The curvature locus of a regular $3$-manifold in $\mathbb{R}^{N}$ can be seen as the image of the
classical Veronese surface of order $2$ via a convenient linear map.

In \cite[p. 34]{Carmen3var}, it is shown that
the curvature locus at a point $p$ where $\dim (N_p^1M^{3}_{\reg})=3$ in a $3$-manifold $M^{3}_{\reg}\subset\mathbb{R}^6$ is isomorphic to one of the following: a Roman Steiner surface (Figure \ref{fig:steiner1}), a Cross-Cap surface (Figure \ref{fig:cross-cap}), a Steiner surface of type $5$ (Figure \ref{fig:tipo5}), a Cross-Cup or type 6 surface (Figure \ref{fig:cross-cup}), an ellipsoid, a (compact) cone or a planar region. The curvature locus at $p$ is said to be substantial if $\dim(E_p)=3$.

\begin{center}
\begin{figure}%[H]
%\begin{center}
\begin{minipage}[b]{0.45\linewidth} \hspace{1.4cm}
\includegraphics[scale=0.35]{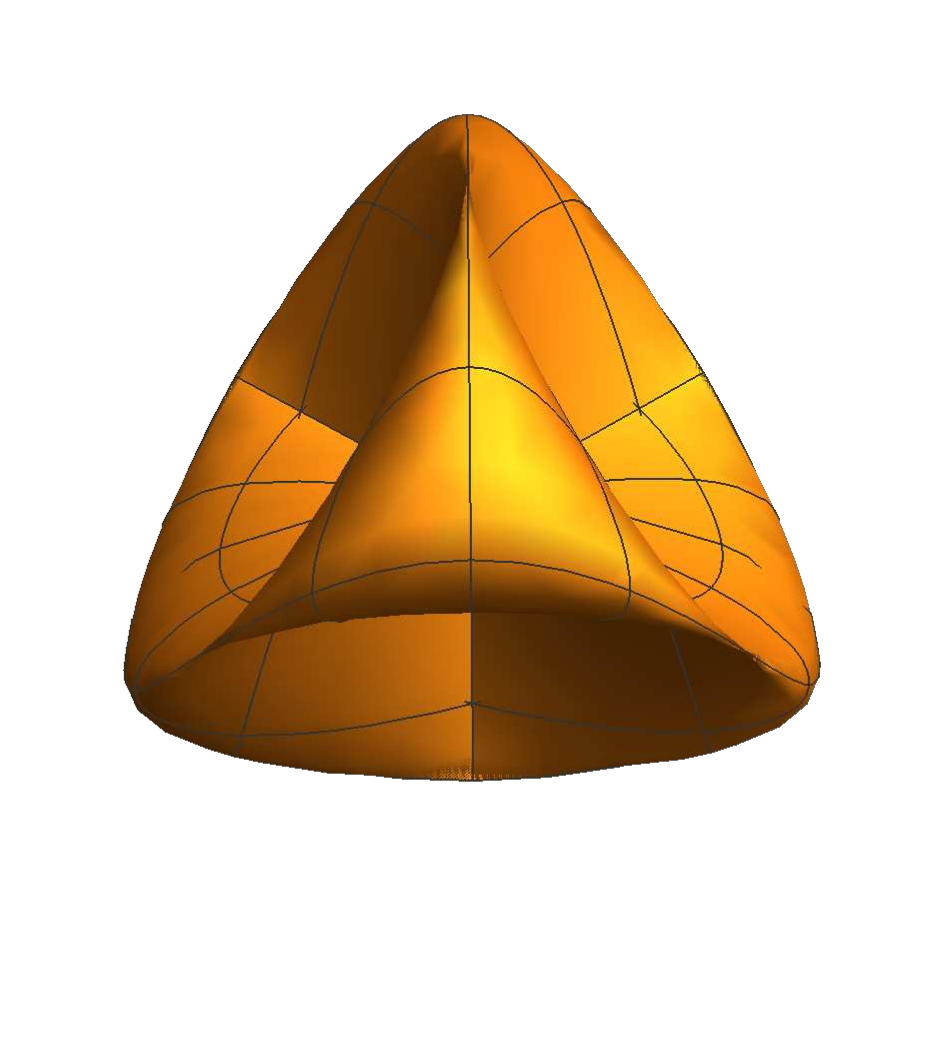}
\caption{Roman Steiner.}
\label{fig:steiner1}
\end{minipage} \hfill
\begin{minipage}[b]{0.45\linewidth} \hspace{1.4cm}
\includegraphics[scale=0.35]{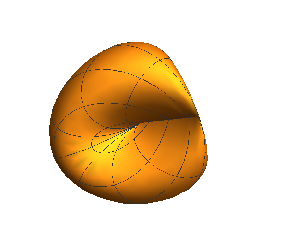}
\caption{Cross-cap.}
\label{fig:cross-cap}
\end{minipage} \hfill \\
%\end{center} \\
\begin{minipage}[b]{0.45\linewidth} \hspace{1cm}
\includegraphics[scale=0.35]{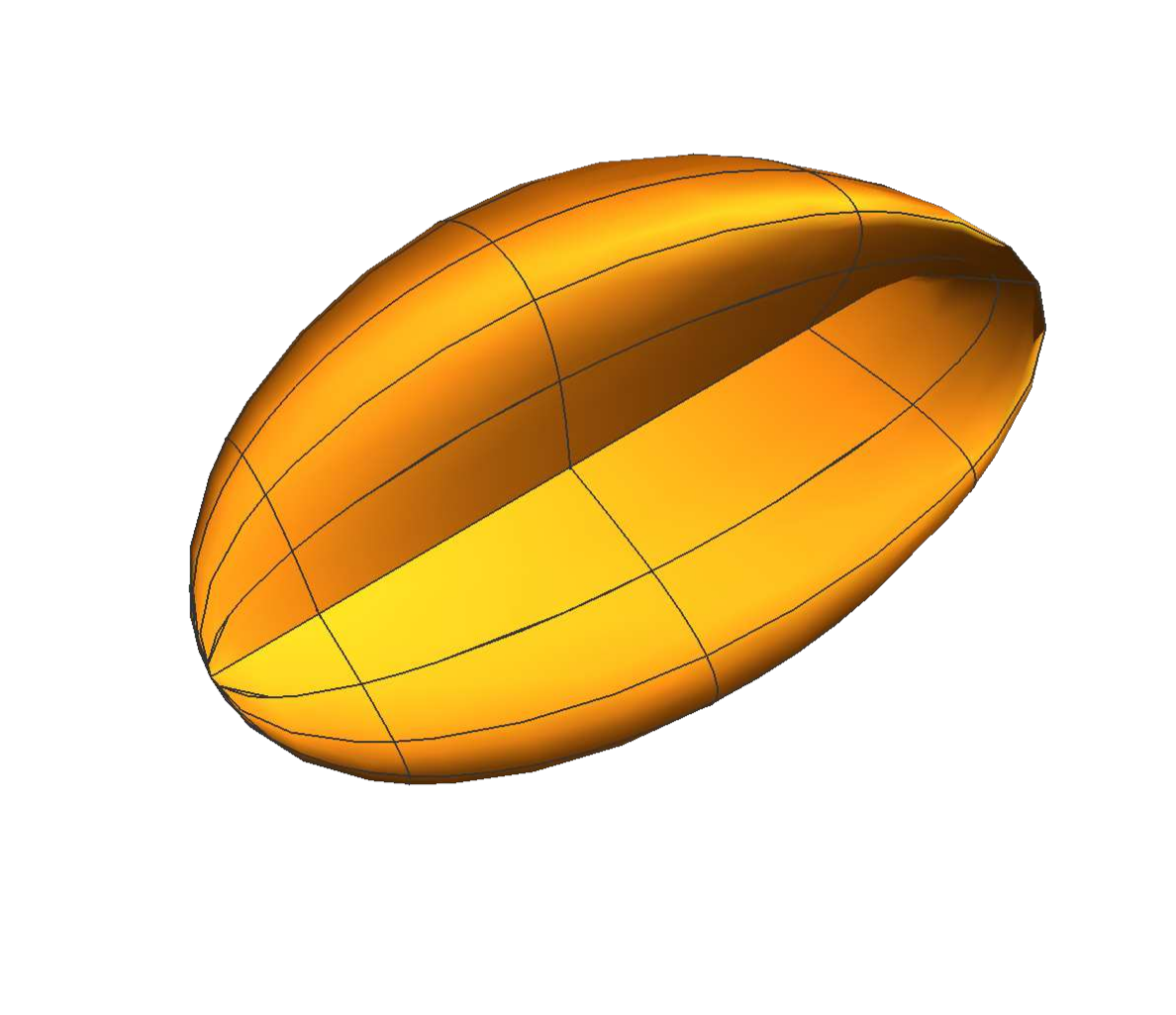}
\caption{Steiner Type $5$.}
\label{fig:tipo5}
\end{minipage} \hfill
\begin{minipage}[b]{0.45\linewidth} \hspace{1.5cm}
\includegraphics[scale=0.35]{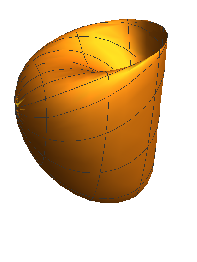}
\caption{Cross-cup.}
\label{fig:cross-cup}
\end{minipage} \hfill
\end{figure}
\end{center}

\textbf{\underline {Singular case}:}

The curvature locus at a singular corank $1$ point $p$ of a $3$-dimensional manifold $M^{3}_{\sing}\subset\mathbb{R}^N$, $N>3$ (denoted by $\Delta_{cv}$) is also given by the image of the map $\eta:C_{q}\rightarrow N_{p}M$ defined by $\eta(u)=II(u,u)$, where $C_{q}\subset T_{q}\tilde{M}$ is the subset of unit tangent vectors (i.e. vectors $u\in T_q\tilde{M}$ such that $I(u,u)^{\frac{1}{2}}=1$).

Examples of topological types of $\Delta_{cv}$ for the case $N=5$ can be found in \cite{BenediniRuasSacramento}.

It is possible to take a coordinate system $\phi$ and make rotations in the target in order to obtain a local parmetrisation for $M^{3}_{\sing}$ at $p$ given by
$$f(x,y,z)=(x,y,f_{3}(x,y,z),\ldots,f_{N}(x,y,z)),$$
where $\frac{\partial f_{i}}{\partial x}(\phi(q))=\frac{\partial f_{i}}{\partial y}(\phi(q))=\frac{\partial f_{i}}{\partial z}(\phi(q))=0$ for $i=3,\ldots,N$. Hence the subset of unit tangent vectors $C_q\in T_q\tilde{M}$ is the cylinder given by
$$C_q=\{(a,b,c)\in T_q\tilde{M}:\ a^2+b^2=1\}.$$ Taking an orthonormal frame $\{\nu_{1},\ldots,\nu_{N-2}\}$ of $N_{p}M^3_{\sing}$, the curvature locus $\Delta_{p}$ can be parametrised by
    $$\eta(a,b,c)=\sum_{i=1}^{N-2}(a^2 l_{\nu_i}+2abm_{\nu_i}+b^2n_{\nu_i}+c^2p_{\nu_i}+2acq_{\nu_i}+2bcr_{\nu_i})\nu_i,$$
where $a^2+b^2=1$,
$$\begin{array}{ccc}
    l_{\nu_i}(q)=\langle \pi_2(f_{xx}),\nu_i\rangle,\,\, & m_{\nu_i}(q)= \langle \pi_2(f_{xy}),\nu_i\rangle,\,\,  & n_{\nu_i}(q)= \langle \pi_2(f_{yy}),\nu_i\rangle,  \\
    p_{\nu_i}(q)= \langle \pi_2(f_{zz}),\nu_i\rangle,\,\, & q_{\nu_i}(q)= \langle \pi_2(f_{xz}),\nu_i\rangle,\,\,  & r_{\nu_i}(q)= \langle \pi_2(f_{yz}),\nu_i\rangle,
  \end{array}
$$
are the coefficients of the second fundamental form with all the partial derivatives evaluated at $\phi(q)$, and $(x,y,z)$ are local coordinates of $\mathbb R^3$.

From now on we will focus in the cases of $M^3_{\reg}\subset \mathbb R^6$ and $M^3_{\sing}\subset \mathbb R^5$. Notice that, locally (i.e. in neighbourhoods of the corresponding points), the latter can be obtained by an orthogonal projection in a tangent direction of the former. Notice too that the second fundamental forms coincide.

For these cases, the definition of asymptotic direction is the following. A unit tangent direction $u$ is an asymptotic direction if there exists a unit normal vector $\nu$ such that $II_{\nu}(u,w)=0$, for any tangent direction $w$. For more details, see Definitions 2.3 and 5.1 in \cite{dreibelbis} and \cite{benediniosetpm}, respectively.

\subsection{Real nets of quadrics}\label{subsec:nets}

The affine geometry of the second fundamental form of a $3$-dimensional manifold in $\mathbb R^6$ is the study of invariants
of the action of the affine group $\mathcal{G}=GL(3)\times GL(3)$ in the space of quadratic polynomial mappings $Q:\mathbb{R}^3\rightarrow\mathbb{R}^3$. Each
such mapping generates a system of quadrics denominated a ``net of quadrics". One can find in the literature many texts about nets of quadrics (see for instance
\cite{wall/duplesis,wall,wall1,wall2}).
%In this subsection we characterize the curvature locus of a smooth 3-manifold in $\mathbb{R}^6$ from the view point of the classification of net of quadrics. These systems were studied by Wall \cite {wall1} and \cite{wall2} in the complex case and by Edwards and Wall \cite{wall} in the real case.
%The classification of the nets of quadrics with respect to the $\mathcal{G}=Gl(3)\times Gl(3)$-equivalence appears in \cite{wall/duplesis}. It can also be found in \cite{wall}.
%One can find in the literature many texts about nets of quadrics. For instance, we can list \cite{wall/duplesis,wall,wall1,wall2}.
In \cite{wall/duplesis} is presented a classification of the real nets of quadrics with respect to $\mathcal{K}$-equivalence (see \cite{mather} or \cite{wall3} for the definition of the contact group and its action).

Let $H^2(3)$ be the space of homogeneous polynomials of degree 2 in 3 variables. A \emph{net of quadrics} is a system in $H^2(3)$ generated by 3 polynomials $q_1,q_2,q_3$, where $q_i\in H^2(3)$ for $i=1,2,3$. Associated to each net there is a map germ $Q:(\mathbb{R}^3,0)\rightarrow (\mathbb{R}^3,0)$, $Q=(q_1,q_2,q_3)$. Denote by $\langle Q \rangle=\langle q_1,q_2,q_3\rangle$ the net given by $Q=(q_1,q_2,q_3)$, that is, $\langle Q\rangle=\{\lambda q_1+\mu q_2+\nu q_3\,|\,\lambda,\mu,\nu\in\mathbb{R}\}$. Let $\Gamma$ be the set of all nets $\langle Q\rangle$. It follows from \cite{wall} that there is a Zariski open set of $\Gamma$, denoted by $\Gamma_0$, such that any net $\langle Q\rangle$ in $\Gamma_0$ can be taken in the form:
\begin{equation}\label{eq:net1}
\lambda(2xz+y^2)+\mu( 2yz)+\nu(-x^2-2gy^2+cz^2+2gxz),\,\,\, c(c+ 9g^2) \neq 0\,\,\, (\mbox{see\,\, \cite{wall4}}).
\end{equation}
For quadratic polynomial mappings $f:\mathbb{C}^3\rightarrow\mathbb{C}^3$, this normal form is equivalent to the Hessian form
\begin{equation}\label{eq:net2}
\lambda(x^2+2cyz)+\mu(y^2+2cxz)+\nu(z^2+2cxy), \,\,\,\,\,c(c^3-1)(8c^3+1)\neq 0.
\end{equation}

A net in this set is called a general real net of quadrics. As the normal forms of the generic nets are given by homogeneous polynomial maps of degree 2, and in this case, the corresponding map germ $Q=(q_1,q_2,q_3)$ is 2-determined with respect to $\mathcal{K}$-equivalence, it follows that the $\mathcal{K}$-classification coincides with
the classification by the action of the group $\mathcal{G}=GL(3)\times GL(3)$.

The complete classification of quadratic mappings $Q=(q_1, q_2, q_3)$ with respect to $\mathcal{G}=GL(3)\times GL(3)$-equivalence can be found
in \cite[p. 315]{wall/duplesis}. The family (\ref{eq:net1}) is labelled $A$, $B$ and $C$ according to the values of the parameters $c$ and $g$.
Table \ref{caseAB} presents the types $A$, $B$ and its subcases. Type $C$ is given by $c=g=0$, and the discriminant for cases $A,B$ and $C$ is
$\Delta=-\lambda^2\nu+(\lambda-2g\nu)(\lambda^2+2g\lambda\nu+(c+g^2)\nu^2)$.

The orbits in the complex case are labelled as follows:
\newline

\begin{tabular}{lccccccccccccccc}
  %\hline
  Name & $A$ & $B$ & $B^*$ & $C$ & $D$ & $D^*$ & $E$ & $E^*$ & $F$ & $F^*$ & $G$ & $G^*$ & $H$ & $I$ & $I^*$ \\
  Codimension & 0 & 1 & 1 & 2 & 2 & 2 & 3 & 3 & 3 & 3 & 4 & 4 & 5 & 7 & 7 \\
  %\hline
\end{tabular}
\newline

The type $A$ depends on a modulus and in the real case it splits four subcases. Type $B$, $B^*$, $D$, $D^*$, $E$, $E^*$, $F$ and $F^*$ also have subcases.
\newline
%\begin{center}

\begin{table}[h]%[tp]
\caption{Orbits $A$ and $B$}
\centering
{\begin{tabular}{cccccc}
  \hline
  $c<-9g^2$ & $c=-9g^2$ & $-9g^2<c<0$ & $c=0$ & $c>0$ &  \\
   & $B_c$ & $A_c$ & $B_a^*$ &  & $g>0$  \\
  $A_b$ &  &  &  & $A_d$ &  \\
   & $B_a$ & $A_a$ & $B_c^*$ &  & $g<0$ \\
  \hline
\end{tabular}
}
\label{caseAB}
\end{table}

The remaining cases are shown in Table \ref{other types}, along with their respective discriminants.

\begin{table}[h]%[tp]
\caption{Other orbits}
\centering
{\begin{tabular}{ccc}
\hline
Name & Normal form & discriminant\\
\hline
$D_a$ &  $\langle x^2,y^2,z^2+2xy\rangle$ & $\nu(\lambda\mu-\nu^2)$ \\
$D_b,\ D_c$ & $\langle x^2-y^2,2xy,x^2\pm z^2\rangle$ & $\nu(\lambda^2+\lambda\nu+\mu^2)$ \\
$D_a^*$ & $\langle 2xz,2yz,z^2+2xy\rangle$ & $\nu(2\lambda\mu-\nu^2)$\\
$D_b^*,\ D_c^*$ & $\langle 2xz,2yz,x^2+y^2\mp z^2 \rangle$ & $\nu(\lambda^2+\mu^2\pm\nu^2)$\\
$E_a,\ E_b$ & $\langle x^2\pm y^2,2xy,z^2 \rangle$ & $\nu(\lambda^2\mp \mu^2)$ \\
$E_a^*,\ E_b^*$ & $\langle x^2\mp y^2,2xz,2yz \rangle$ & $\lambda(\mu^2\pm\nu^2)$  \\
$F_a,\ F_b$ & $\langle x^2\pm y^2,2xy,2yz \rangle$ & $\lambda\nu^2$  \\
$F_a^*,\ F_b^*$ & $\langle x^2\mp y^2,2xz,z^2 \rangle$ &  $\lambda(\lambda\nu-\mu^2)$ \\
$G$ &  $\langle x^2,y^2,2yz \rangle$ & $\lambda\nu^2$ \\
$G^*$ & $\langle 2xy,2xz,z^2 \rangle$ & $\lambda^2\nu$ \\
$H$ & $\langle x^2,2xy,y^2+2xz \rangle$ & $\nu^3$ \\
$I$ & $\langle x^2,2xy,y^2 \rangle$ & $0$ \\
$I^*$ & $\langle 2xz,2yz,z^2 \rangle$  & $0$ \\
\hline
\end{tabular}
}
\label{other types}
\end{table}

%\subsection{Determinantal varieties}

%Somente o necess\'{a}rio.

\section{Curvature loci for regular $3$-manifolds}

The curvature locus is the image under a homogeneous quadratic map of the unitary tangent directions in $T_pM^3_{\reg}$, namely the image of $\eta:\mathbb{S}^{2}\subset T_{p}M^{3}_{\reg}\rightarrow N_{p}M^{3}_{\reg}\simeq \mathbb R^3$.
The group $\mathcal A=\mathcal R\times \mathcal L,$ where $\mathcal R=\{ h: \mathbb{S}^{2} \to \mathbb{S}^{2},\,\, h \, \text{diffeomorphism}\}$ and
$\mathcal L=\{ k: \mathbb{R}^{3} \to \mathbb{R}^{3},\,\, k \, \text{diffeomorphism}\}$ acts on the space of maps $\eta : \mathbb{S}^{2} \to \mathbb{R}^{3},$
with the Whitney topology. If two maps $\eta_1$ and $\eta_2$ are $\mathcal A$-equivalent, their images are diffeomorphic. In a neighbourhood of a point in $\mathbb S^2,$ one can choose coordinates such that $\eta$  can locally be seen as a map from $\mathbb R^2$ to $\mathbb R^3$. From the point of view of $\mathcal A$-equivalence the only stable singularity (stable under small perturbations) is the cross-cap ($CC$), for which we have a standard normal form given by $(x,y)\mapsto (x,y^2,xy).$

At a singular point $p,$ we say that $\eta : (\mathbb R^2,p) \to (\mathbb R^3, \eta(p))$ is  finitely determined if there exists a positive integer $k $ such that
 for any
 $\eta' : (\mathbb R^2,p) \to (\mathbb R^3, \eta(p)),$ with $j^k\eta(p)=j^k\eta'(p),$ it follows that $\eta$ and $\eta'$ are $\mathcal A$-equivalent at $p.$
 %  it is determined up to $\mathscr A$-equivalence by a certain $k$-jet for some finite $k$.
 By the Mather-Gaffney criterion, a singularity $\eta : (\mathbb R^2,p) \to (\mathbb R^3, \eta(p))$ is not finitely determined if it is not stable outside the point $p$ (see \cite{libroNunoMond}).

In \cite{computeiros}  the projective classification of quadratically parametrizable surfaces is given. Out of the surfaces studied there the compact ones are the possibilities for curvature loci of regular 3-manifolds, as pointed out in \cite{Carmen3var}. The description of these surfaces, given in \cite{computeiros} in a different terminology, is the following:

\begin{itemize}
\item[(i)] The Roman Steiner surfaces has 6 cross-cap singularities ($CC$-points) joined in pairs by 3 transversal double point curves which intersect at a triple point.
\item[(ii)] The Cross-cap surface has 2 cross-cap singularities joined by a transversal double point curve. From the geometrical point of view, generically one of these cross-caps is elliptic and the other one hyperbolic (see \cite{BallesterosTari} or \cite{OsetSinhaTari}).
\item[(iii)] The type 5 surface has 2 cross-caps joined by a transversal double point curve. One of these cross-caps lies on a tangent double point curve. The tangency of this double point curve is non-degenerate (the tangent sheets have different curvature) and the two end points of this curve are non finitely determined singularities. We shall call the two end points of the tangent double point curve $TCC$-points (for tangent cross-cap).
\item[(iv)] The type 6 surface has no cross-caps. It has a tangent double point curve where the tangency is degenerate (the tangent sheets have the same curvature). The two end points of this curve are non finitely determined singularities. We shall call the two end points of the tangent double point curve $DTCC$-points (for degenerate tangent cross-cap).
\end{itemize}

Besides these, amongst the non-planar loci we have

\begin{itemize}
\item[(v)] The truncated cone is a compact cone with a curve of singular points corresponding to the base of the cone and a singular point corresponding to the vertex of the cone. None of these singularities are finitely determined.
\item[(vi)] The ellipsoid has a curve of non finitely determined singular points.
\end{itemize}

From this classification it follows that the number and type of singularities of quadratically parametrised surfaces determine completely its equivalence class. It is also clear that generically one would expect to obtain a Roman Steiner surface or a Cross-cap surface. Besides this, following \cite{computeiros, Reid}, the $TCC$-points and the $DTCC$-points (i.e. the non-finitely determined singularities of types 5 and 6) can be distinguished by the multiplicity of the double line in a neighbourhood of them, which is determined by intersecting the double line with transverse
planes and resolving the singularity of the resulting curve by blowing up. We say that a self intersection is of multiplicity 1 if the two intersecting sheets
are transverse, of multiplicity 2 if the sheets are tangent but have
different curvature and of multiplicity 3 if the sheets are tangent and have the same
curvature. With this notation, the cross-cap has multiplicity 1, the $TCC$-points have multiplicity 2 and the $DTCC$-points have multiplicity 3.

The map $\eta$ is the restriction of the second fundamental form $II$ to $\mathbb{S}^2=\rho^{-1}(0)$, where $\rho(x,y,z)=x^2+y^2+z^2-1$ and $x,y,z$ are the local coordinates of $T_pM^3_{\reg}$ and of $\mathbb R^3$. The singularities of this restriction are captured by the zeros of the 3x3 minors of the following determinantal matrix. Given a local parametrisation in Monge form $f(x,y,z)=(x,y,z,f_1(x,y,z),f_2(x,y,z),f_3(x,y,z))$ ($f_1,f_2,f_3$ have zero linear and constant part) of $M^{3}_{\reg}$ where $p$ is the origin, define the matrix
$$M_f=\left(
  \begin{array}{ccc}
    \frac{\partial j^2f_1}{\partial x} & \frac{\partial j^2f_1}{\partial y} & \frac{\partial j^2f_1}{\partial z} \\
    \frac{\partial j^2f_2}{\partial x} & \frac{\partial j^2f_2}{\partial y} & \frac{\partial j^2f_2}{\partial z} \\
    \frac{\partial j^2f_1}{\partial x} & \frac{\partial j^2f_1}{\partial y} & \frac{\partial j^2f_1}{\partial z} \\
    \frac{\partial \rho}{\partial x} & \frac{\partial \rho}{\partial y} & \frac{\partial \rho}{\partial z} \\
  \end{array}
\right).$$
The zeros of the 3x3 minors of this matrix, $V(M_f)$, are the solutions to a system of 4 homogeneous polynomials of degree 3 in 3 variables, i.e. 4 ternary cubics. The first one of these equations is the determinant of the 3x3 minor given by the 3 first rows of $M_f$ and is the determinant of the Jacobian matrix of the second fundamental form, which we call $\delta$. This first equation is precisely the equation to obtain the asymptotic directions (see \cite{dreibelbis}). We call $\delta_i$, $i=1,2,3$ the determinant of the remaining 3 minors, where $i$ is the row of $M_f$ removed. The solution to this system is a homogeneous algebraic variety which is generically a collection of lines passing through the origin. These solutions have a certain multiplicity defined by the number of points of the intersection of the line with a generic plane away from the origin. The intersection of these lines with $\mathbb{S}^2$ correspond to the singularities of the curvature locus, so we have the following characterisation.

\begin{teo}\label{class1}
Suppose $M^{3}_{\reg}$ is parametrised by $f$ in Monge form and suppose that it has a substantial (i.e. non-planar) curvature locus $\Delta_v$. Then
\begin{itemize}
\item[(i)] $\Delta_v$ is a Roman Steiner surface if and only if $V(M_f)$ is 6 real lines of multiplicity 1.
\item[(ii)] $\Delta_v$ is a Cross-cap surface if and only if $V(M_f)$ is 2 real lines and 4 complex lines of multiplicity 1.
\item[(iii)] $\Delta_v$ is a type 5 surface if and only if $V(M_f)$ is 4 real lines, two of them with multiplicity 2 and 2 with multiplicity 1.
\item[(iv)] $\Delta_v$ is a type 6 surface if and only if $V(M_f)$ is 2 real lines of multiplicity 3 each.
\item[(v)] $\Delta_v$ is a truncated cone if and only if $V(M_f)$ has a plane and a real line as the only real solutions and (possibly) some complex solutions.
\item[(vi)] $\Delta_v$ is an ellipsoid if and only if $V(M_f)$ has a plane as the only real solution and (possibly) some complex solutions.
\end{itemize}
\end{teo}
\begin{proof}
This follows from the description of the types of locus and their singularities and the relation between the multiplicity of the singular points in the locus defined above and the multiplicity of the solutions in $V(M_f)$. Items i) to iv) are a consequence of the following fact: the complexification of $M_f$ is the representation matrix of a codimension 2 Cohen-Macaulay determinantal singularity in $\mathbb C^3$ (see \cite{F-KN}, \cite{FK}). The equations defining the minors of $M_f$ are homogeneous of degree 3,
so the ideal defined by $\delta, \delta_1, \delta_2, \delta_3$ has multiplicity 6 (see Lemma 5.5 in \cite{bivianuno}) (i.e. the homogeneous curve intersects in 6 points a plane away from the origin) and the sum of the multiplicities of the solutions to the system must be 6. A solution to this system is generically a real line with multiplicity 1. On the other hand, in the curvature locus a singularity is generically a cross-cap. This means that the real lines of multiplicity 1 correspond to $CC$-points, i.e. the multiplicity of the solution to the system coincides with the multiplicity defined for $CC$-points.

Now, if we deform a $TCC$-point, two $CC$-points appear generically. By conservation of number, this means that the $TCC$-points correspond to solutions of multiplicity 2. Similarly, the $DTCC$-points deform into three $CC$-points, and so $DTCC$-points correspond to real lines of multiplicity 3 as solution to the system. Taking into account the multiplicities of the singular points for each type of locus and that the sum of the multiplicities must be 6, the only possibilities are as in the statement.

For items (v) and (vi) $V(M_f)$ is not a determinantal variety since there are mixed dimensions and the multiplicity is not well defined.
\end{proof}

\begin{rem}
There is a natural equivalence when working with singularities of matrices. The group $\bar{\mathcal G}=\mathcal R\times\mathcal H$ where $\mathcal R$ is changes of coordinates in the source and $\mathcal H=GL(4)\times GL(3)$ acts naturally by multiplication to the left and to the right on the $4\times 3$ matrices. If two matrices $A$ and $B$ are $\bar{\mathcal{G}}$-equivalent, then $V(A)$ and $V(B)$ are isomorphic, where $V$ stands for the zeroes of the $3\times3$ minors. This means that inside a certain $\bar{\mathcal G}$-orbit, the affine type of curvature locus of the associated 3-manifold does not change.

However, two parametrisations $f,g$ of regular 3-manifolds whose associated matrices $M_f$ and $M_g$ are $\bar{\mathcal G}$-equivalent, may have associated nets of quadrics in different affine $\mathcal G$-orbits, as examples (iii) and (iv) below show.
\end{rem}

\begin{ex}
Consider $M^{3}_{\reg}\subset\mathbb{R}^6$ given by $f$, a local parametrisation at the origin $p$ in a Monge form, as before.
\begin{itemize}
    \item[(i)] Let $G(x,y,z)=(2xy,2xz,z^2)$ be its second fundamental form at $p$. The 3x3 minors of $M_f$ are $\delta=-8xz^2$, $\delta_1=-8yz^2$, $\delta_2=-8z(y^2-x^2)$ and $\delta_3=8x(x^2-z^2-y^2)$. Hence, the solution $V(M_f)$ of $\delta=\delta_i=0$, $i=1,2,3$ is given by $\{(\pm y,y,0)\}\cup\{(0,y,0)\}\cup\{(0,0,z)\}$, $y,z\in\mathbb{R}$. Take the solution $\{(0,0,z)\}$, and consider the plane $\{z=1\}$, which is transverse to the line. The algebraic multiplicity of this line is given by replacing $z$ by 1 and evaluating in $(x,y)=(0,0)$ the dimension of the local algebra $\dim_{\mathbb C}\frac{\mathcal O_3}{\langle \delta,\delta_1,\delta_2,\delta_3\rangle}$, where $\mathcal{O}_3$ is the local ring of functions in 3 variables. Hence, we obtain multiplicity 1 in this case. Similarly, the first two solutions have multiplicity 2 and correspond to $TCC$-points and the last two solutions correspond to the cross-cap points. Here $\Delta_v$ is a type 5 surface.
    \item[(ii)] Now, let $f$ be such that the second fundamental form at $p$ is given by $G(x,y,z)=(x^2,2xy,y^2+2xz)$. For this case, $\delta=8x^3$,
    $\delta_1=8(z-x)(y^2-x^2-xz)$, $\delta_2=8xy(z-x)$ and $\delta_3=8x^2z$. There are two solutions for $\delta=\delta_i=0$, $i=1,2,3$: $\{(0,y,0)\}\cup\{(0,0,z)\}$, $y,z\in\mathbb{R}$. Both of these solutions have multiplicity 3 and correspond to $DTCC$-points. Here $\Delta_v$ is a type 6 surface.
    \item[(iii)] Taking $G(x,y,z)=(x^2,2xy,y^2)$ as the second fundamental form at $p$, $\delta=0$, $\delta_1=8y^2z$, $\delta_2=8xyz$ and $\delta_3=8x^2z$. Thus, the solution of $\delta=\delta_i=0$, $i=1,2,3$, is $\{(x,y,0)\}\cup\{(0,0,z)\}$, $x,y,z\in\mathbb{R}$. Notice that the solution is a plane and a line. The curvature locus is a truncated cone. The associated net of quadrics lies in orbit $I$ from Table \ref{other types}.
    \item[(iv)] Finally, let $G(x,y,z)=(-x^2-y^2+2z^2,\frac{1}{2}x^2-\frac{1}{2}y^2,xz)$. For this example, $\delta=4y(x^2+z^2)$, $\delta_1=2y(z^2-2x^2)$, $\delta_2=12yz^2$ and $\delta_3=24xyz$. Therefore $V(M_f)=\{(x,0,z)\}\cup\{(0,y,0)\}$, $x,y,z\in\mathbb{R}$, that is,
    a plane and a line and $\Delta_v$ is also a truncated cone. Here the associated net of quadrics lies in orbit $F_a^*$ from Table \ref{other types}.

\end{itemize}

%POR UM EXEMPLO DE CONTAS PRA TIPO 5 ($G^*$) E OUTRO PRA TIPO 6 ($H$) ONDE D�? PRA VER AS MULTIPLICIDADES DAS SOLUÇOES. SE QUISER, POR 2 EXEMPLOS DO CONE ONDE A RETA OU O PLANO TEM DIFERENTES MULTIPLICIDADES ($I$ E $(-x^2-y^2+2z^2,\frac{1}{2}x^2-\frac{1}{2}y^2,xz)$)(A RETA PODE TER MULTIPLICIDADE 4 OU PODE TER VARIAS RETAS, SO QUE UMA DELAS DENTRO DO PLANO).
\end{ex}

\begin{rem}
As a corollary of this theorem we get certain implications about the solutions of the system of 4 ternary cubics described above. For example, the fact that there is always a real solution means that at least one of the four minors must be a reducible equation. We can also see that the complex solutions cannot have multiplicity 2, as that would imply that one of the minors is a degree 4 polynomial (complex solutions come together with their conjugate solutions) and therefore would not be a cubic.
\end{rem}

This characterisation of the topological types of non-planar curvature loci allows us to obtain sufficient geometrical conditions in order to have a certain topological type.

\begin{prop}
Let the curvature locus of $M^{3}_{\reg}$
at $p$ be parametrised by
$\eta:\mathbb{S}^{2}\subset T_pM^{3}_{\reg}\rightarrow N_pM^{3}_{\reg},\ (\theta,\phi)\mapsto\eta(\theta,\phi)$,
where
$$
\begin{array}{cl}
    \eta(\theta,\phi)= & H+(1+3\cos(2\phi))B_1+\cos(2\theta)\sin^2\theta B_2 \\
     & +\sin(2\theta)\sin^2\phi B_3+cos\theta\sin(2\phi)B_4+\sin\theta\sin(2\phi)B_5
\end{array}
$$
with
$$
\begin{array}{c}
     H=\frac{1}{3}(f_{xx}+f_{yy}+f_{zz}),\ B_1=\frac{1}{12}(-f_{xx}-f_{yy}+2f_{zz}),  \\
     B_2=\frac{1}{2}(f_{xx}-f_{yy}),\ B_3=f_{xy},\ B_4=f_{xz},\ B_5=f_{yz}.
\end{array}
$$
If $H=B_1=B_2=(0,0,0)$ (i.e. $f_{xx}=f_{yy}=f_{zz}=\bar{0}$) and $B_3,B_4,B_5$ are linearly independent, then the curvature locus is a Roman Steiner surface.
\end{prop}
\begin{proof}
Let $B_3=(b_1,b_2,b_3)$, $B_4=(c_1,c_2,c_3)$ and $B_4=(r_1,r_2,r_3)$. Call $R=\det(B_3,B_4,B_5)$. We have $\delta=-2Rxyz$. Since $R\neq 0$, to calculate $V(M_F)$ $x,y$ or $z$ must be zero. Suppose $x=0$, then $$\delta_1=2(z^2-y^2)(\left|
                                                                                                                                                                   \begin{array}{cc}
                                                                                                                                                                     b_1 & r_1 \\
                                                                                                                                                                     b_2 & r_2 \\
                                                                                                                                                                   \end{array}
                                                                                                                                                                 \right|y+\left|
                                                                                                                                                                   \begin{array}{cc}
                                                                                                                                                                     c_1 & r_1 \\
                                                                                                                                                                     c_2 & r_2 \\
                                                                                                                                                                   \end{array}
                                                                                                                                                                 \right|z)$$
$$\delta_2=2(z^2-y^2)(\left|\begin{array}{cc}
                                                                                                                                                                     b_1 & r_1 \\
                                                                                                                                                                     b_3 & r_3 \\
                                                                                                                                                                   \end{array}
                                                                                                                                                                 \right|y+\left|
                                                                                                                                                                   \begin{array}{cc}
                                                                                                                                                                     c_1 & r_1 \\
                                                                                                                                                                     c_3 & r_3 \\
                                                                                                                                                                   \end{array}
                                                                                                                                                                 \right|z)$$
$$\delta_3=2(z^2-y^2)(\left|\begin{array}{cc}
                                                                                                                                                                     b_2 & r_2 \\
                                                                                                                                                                     b_3 & r_3 \\
                                                                                                                                                                   \end{array}
                                                                                                                                                                 \right|y+\left|
                                                                                                                                                                   \begin{array}{cc}
                                                                                                                                                                     c_2 & r_2 \\
                                                                                                                                                                     c_3 & r_3 \\
                                                                                                                                                                   \end{array}
                                                                                                                                                                 \right|z)$$
Notice that $R\neq 0$ implies that at most 3 2x2 minors can have 0 determinant, so at most one of the above equations vanishes. Therefore, the only solution different from $x=y=z=0$ is $y=\pm z$.

Similarly when $y=0$ or $z=0$. Therefore $V(M_f)$ is 6 real lines whose intersection with $\mathbb{S}^2$ is the 6 cross-caps of the Roman Steiner surface.
\end{proof}

\section{Curvature loci for singular 3-manifolds}

Projecting orthogonally a regular 3-manifold in $\mathbb R^6$ along a tangent direction yields a singular 3-manifold in $\mathbb R^5$. The second fundamental form is the same in both cases. However, the curvature locus is different since the unitary tangent vectors form an $\mathbb{S}^2$ in the regular case and the cylinder $C_q$ in the singular case.

In \cite[p. 402]{benediniosetpm} asymptotic directions for $M^3_{\sing}\subset\mathbb R^5$ are defined and studied, in particular, the authors show that a direction is asymptotic if and only if the determinant of the Jacobian of the second fundamental form vanishes. They also show that, similarly to the regular case (\cite[p. 451]{dreibelbis}), the singular points of the curvature locus correspond to the image of asymptotic directions by the second fundamental form. Furthermore, they show that when you project orthogonally a regular 3-manifold in $\mathbb R^6$ along an asymptotic direction this direction becomes an infinite asymptotic direction of the singular 3-manifold in $\mathbb R^5$. This suggests that when you consider $M^3_{\sing}\subset\mathbb R^5$ seen as the projection of $M^3_{\reg}\subset\mathbb R^6$, the number of singularities of the curvature locus will depend on whether the direction of projection is asymptotic or not. Our goal in this section is to understand this situation.

Consider $C_q=h^{-1}(0)$ where $h(x,y,z)=x^2+y^2-1$. Changing the last row of $M_f^{\reg}$ by the partials of $h$ we get a determinantal matrix which we call $M_f^{\sing}$. The singularities of the curvature locus of a singular 3-manifold are controlled by $V(M_f^{\sing})$. We call the determinants of the minors of $M_f^{\sing}$ $\delta^{\sing}$ and $\delta_{i}^{\sing}$, $i=1,2,3$. We use the superscript $\reg$ for the determinants in the regular case. Notice that $\delta^{\reg}=\delta^{\sing}$.

In contrast with the regular case, the classification of the topological types of curvature loci for singular 3-manifolds in $\mathbb R^5$ seems very hard to tackle as the following examples illustrate. Consider $M^{3}_{\reg}\subset\mathbb{R}^6$ given by $f$, a local parametrisation at the origin $p$ in a Monge form. We shall project $M^3_{\reg}$ along different tangent directions to showcase the possibilities for the curvature locus of the projected singular 3-manifold $M^3_{\sing}\in \mathbb R^5$ at the origin.

\begin{ex}\label{ExampleProjections} Let $Q(x,y,z)=(\frac{1}{2}x^2-\frac{1}{2}y^2,xz,yz)$ be the second fundamental form of $M^3_{\reg}$ at $p$. The curvature locus
    at the origin is a Roman Steiner surface, since $V(M_f)$ is 6 real lines. Consider the asymptotic direction $v=(0,0,1)$, whose
    image by $G$ is the triple point (hence it is not a cross-cap point). The curvature locus at the origin of $M^3_{\sing}$ obtained by orthogonally projecting along $v$ is given by
    $(\cos(2\theta),\frac{2\cos\theta\cos\phi}{\sin\phi},\frac{2\sin\theta\cos\phi}{\sin\phi})$,
    a surface with two cross-cap points (there are 3 solutions in $V(M_f^{\sing})$, one being the $z$-axis, the null tangent direction) that is shown in Figure \ref{FigureExample} (left). Using the same map $G$, we shall project along $v=(0,1,0)$, an asymptotic direction whose image is a cross-cap. In order to do so, we make a rotation $T(x,y,z)=(x,z,-y)$ in the source, taking $v$ to $(0,0,1)$, so that $(G\circ T)(x,y,z)=(\frac{1}{2}x^2-\frac{1}{2}z^2,-xy,-yz)$. Hence, $V(M_f^{\sing})$ has six solutions, one being the $z$-axis and the curvature locus at the origin is a surface with five cross-caps points, as seen in Figure \ref{FigureExample} (center). Still considering $G$ and projection along $v=(\frac{\sqrt{2}}{4},\frac{\sqrt{2}}{4},\frac{\sqrt{3}}{2})$, an asymptotic direction whose image lies in a double point curve, we obtain again six solutions in $V(M_f^{\sing})$, one being the $z$-axis. Again, the curvature locus has five cross-caps (see Figure \ref{FigureExample}, right). Rotations in the source are also applied to take $v$ to $(0,0,1)$, and we obtain $(G\circ T_1)(x,y,z)=(-x(\frac{\sqrt{3}}{2}y-\frac{1}{2}z),\frac{\sqrt{2}}{2}(x-\frac{\sqrt{3}}{2}y+\frac{1}{2}z)(\frac{1}{2}y+\frac{\sqrt{3}}{2}z),\frac{\sqrt{2}}{2}(x+\frac{\sqrt{3}}{2}y-\frac{1}{2}z)(\frac{1}{2}y+\frac{\sqrt{3}}{2}z))$.
\begin{comment}    
    \begin{figure}[!htb]
%\begin{center}
\begin{minipage}[b]{0.3\linewidth} \hspace{1.9cm}
\includegraphics[scale=0.3]{figEx4.1a.png}
\caption{\small{First projection.}}
\label{fig:exa}
\end{minipage} %\hfill
\begin{minipage}[b]{0.3\linewidth} \hspace{-0.75cm}
\includegraphics[scale=0.3]{figEx4.1b.png}
\caption{Second projection.}
\label{fig:exb}
\end{minipage} %\hfill \\
%\end{center} \\
\begin{minipage}[b]{0.3\linewidth} \hspace{1.6cm}
\includegraphics[scale=0.3]{figEx4.1c.png}
\caption{Third projection.}
\label{fig:exc}
\end{minipage} \hfill
\end{figure}
\end{comment}

	\begin{figure}[h!]
		\begin{center}
			\includegraphics[scale=0.4]{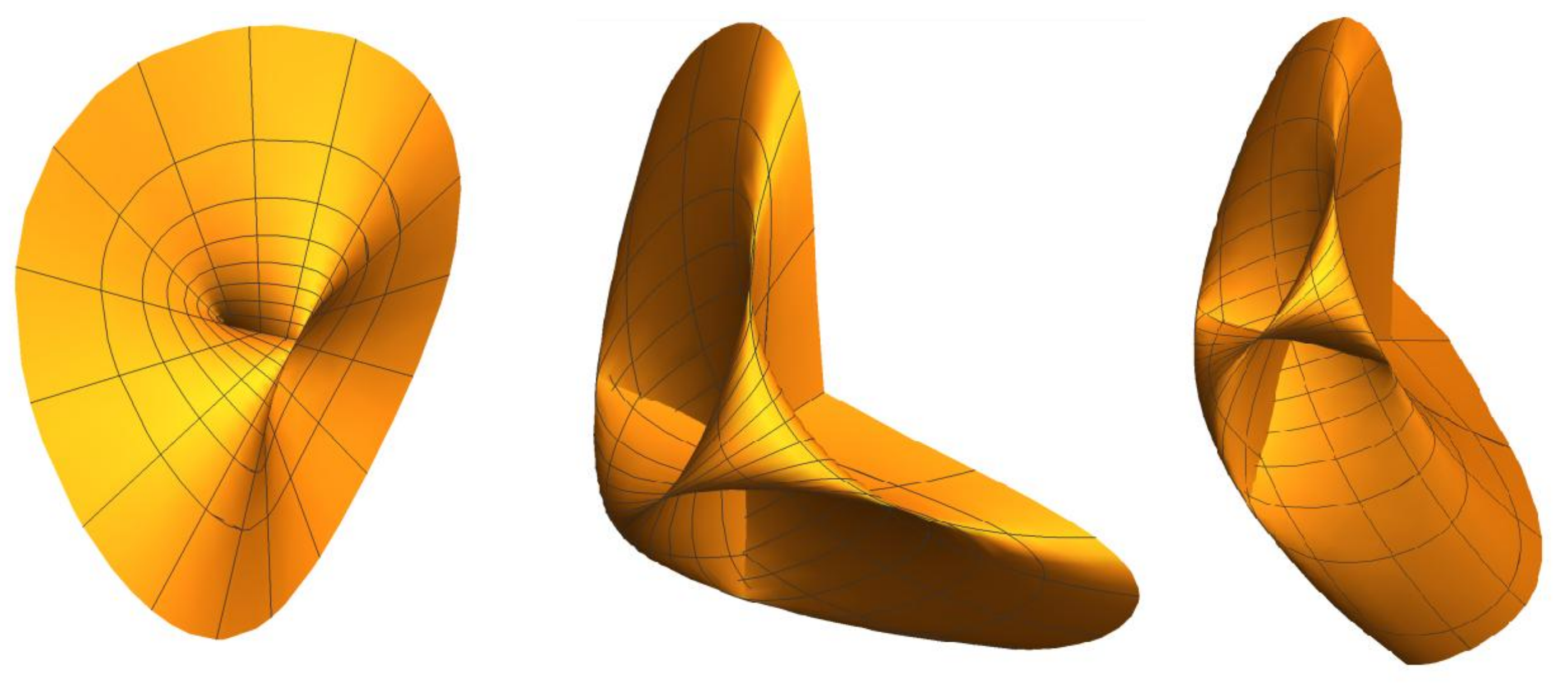}
			\caption{Curvature loci of projections in Example \ref{ExampleProjections}.}
			\label{FigureExample}
		\end{center}
	\end{figure}

\end{ex}

\begin{ex}
 Consider $Q(x,y,z)=(x^2+yz,y^2+xz,z^2+xy)$ as the second fundamental form of $M^3_{\reg}$ at $p$. The curvature locus at $p$ is a Roman Steiner surface, but this time the $z$-axis is not an asymptotic direction. Projecting along $v=(0,0,1)$, $V(M_f^{\sing})$ has six solutions, none of them being the $z$-axis, which means that the curvature locus at the origin of the singular 3-manifold has six cross-caps and it is parametrised by $(2\cos^2\theta+2\sin\theta\frac{\cos\phi}{\sin\phi},2\sin^2\theta+2\cos\theta\frac{\cos\phi}{\sin\phi},\frac{\cos^2\phi}{\sin^2\phi}+\sin(2\theta))$. Finally, we will project along $v=(\frac{\sqrt{2}}{2},0,\frac{\sqrt{2}}{2})$, whose image is the triple point of the curvature locus. Rotating the source to take $v$ to $(0,0,1)$, we obtain $(G\circ T)(x,y,z)=(\frac{1}{2}(x-z)^2+\frac{\sqrt{2}}{2}(x+z)y,y^2+\frac{1}{2}(x^2-z^2),\frac{1}{2}(x+z)^2+\frac{\sqrt{2}}{2}(x-z)y)$. In this case, $V(M_f^{\sing})$ has five solutions, one of them being the $z$-axis. Therefore, the curvature locus at the origin has four cross-caps.
\end{ex}

%POR EXEMPLO DE ROMAN STEINER E DIFERENTES PROJEÇOES TAL QUE UMA TEM 6 CROSS-CAPS, OUTRA 5, OUTRA 4, OUTRA 2, ETC. PRA MOSTRAR O ZOO DE POSSIBILIDADES. NAS CONTAS QUE FIZEMOS NO RIO TEM EXEMPLOS.

\begin{teo}\label{iso}
Consider $M^3_{\reg}\subset\mathbb R^6$ and its projection along a tangent direction $M^3_{\sing}\subset\mathbb R^5$. Suppose that the direction of projection is not asymptotic. Then $V(M_f^{\reg})$ is isomorphic to $V(M_f^{\sing})$. In particular, the curvature loci of $M^3_{\reg}\subset\mathbb R^6$ and $M^3_{\sing}\subset\mathbb R^5$ have the same number and type of singularities.
\end{teo}
\begin{proof}
For simplicity we consider the direction of projection to be $(0,0,1)$ since by rotation in the tangent space we can take any other direction to $(0,0,1)$. Express the 2-jet of the parametrisation of the singular 3-manifold as
$$(x,y,q_1(x,y)+a_{21}xz+a_{22}yz+a_{12}z^2,q_2(x,y)+b_{21}xz+b_{22}yz+b_{12}z^2,q_3(x,y)+c_{21}xz+c_{22}yz+c_{12}z^2),$$ where $q_1,q_2,q_3$ are homogeneous degree 2 polynomials in $x,y$. Evaluating the Jacobian of the second fundamental form on $(0,0,1)$ we get the matrix
$$\alpha'=\left(
  \begin{array}{ccc}
    a_{21} & a_{22} & 2a_{12} \\
    b_{21} & b_{22} & 2b_{12} \\
    c_{21} & c_{22} & 2c_{12} \\
  \end{array}
\right).$$ If $(0,0,1)$ is not asymptotic, then $\det(\alpha')\neq 0$. By Proposition 5.2 in \cite{BenediniRuasSacramento} if $\det(\alpha')\neq 0$ then, by linear changes of coordinates in the target, the 2-jet of the parametrisation can be taken to $$f(x,y)=(x,y,P_1(x,y)+ayz,P_2(x,y)+bxz,P_3(x,y)+cz^2),$$ with $abc\neq 0$ and such that $(P_1,P_2,P_3)$ is $GL(2)\times GL(3)$-equivalent to one of the orbits $(x^2,y^2,xy),(x^2,y^2,0),(xy,x^2-y^2,0),(x^2,xy,0),(x^2\pm y^2,0,0),(x^2,0,0)$ or $(0,0,0)$. We consider now the matrix
$$M_f^{\reg}=\left(
  \begin{array}{ccc}
    \frac{\partial P_1}{\partial x} & \frac{\partial P_1}{\partial y}+az & ay \\
    \frac{\partial P_2}{\partial x}+bz & \frac{\partial P_2}{\partial y} & bx \\
    \frac{\partial P_3}{\partial x} & \frac{\partial P_3}{\partial y} & 2cz \\
    2x & 2y & 2z \\
  \end{array}
\right).$$
%The group $\bar{\mathcal G}=\mathcal R\times\mathcal H$ where $\mathcal R$ is changes of coordinates in the source and $\mathcal H=GL(4)\times GL(3)$ acts naturally by multiplication to the left and to the right, acts on the $4\times 3$ matrices in a way that if two matrices $A$ and $B$ are $\bar{\mathcal{G}}$-equivalent, then $V(A)$ and $V(B)$ are isomorphic, where $V$ stands for the zeroes of the $3\times3$ minors. 

Using $\mathcal H$ from the group $\bar{\mathcal G}=\mathcal R\times\mathcal H$ we can do operations with the lines of $M_f^{\reg}$ to obtain a matrix $\tilde M$ which differs from $M_f^{\reg}$ only in the last line, where we get $(2x-c\frac{\partial P_3}{\partial x},2y-c\frac{\partial P_3}{\partial y},0)$. Doing the change of coordinates $X=2x-c\frac{\partial P_3}{\partial x},Y=2y-c\frac{\partial P_3}{\partial y},Z=z$ we obtain a matrix $\tilde{\tilde{M}}$ such that $V(M_f^{\reg})$ is isomorphic to $V(\tilde{\tilde{M}})$. Notice that $V(\tilde{\tilde{M}})$ gives the singular set of a map $h$ restricted to the cylinder $C_q$ where $h$ is $(P_1(x,y)+ayz,P_2(x,y)+bxz,P_3(x,y)+cz^2)$ composed with the previous change of coordinates in the source. This change of coordinates does not affect the corresponding image, therefore, $V(\tilde{\tilde{M}})$ is ismorphic $V(M_f^{\sing})$ and so $V(M_f^{\reg})$ is isomorphic to $V(M_f^{\sing})$.
\end{proof}

\begin{prop}
Suppose the direction of projection is an asymptotic direction, then the curvature locus of the singular manifold has at least one singularity less than the locus of the regular manifold. In particular, if the direction of projection corresponds to a cross-cap singularity, this cross-cap goes to infinity in the locus of the singular manifold.
\end{prop}
\begin{proof}
If $(0,0,1)$ is asymptotic then $\delta^{\reg}(0,0,z)=0$ and so $\delta^{\sing}(0,0,z)=0$. On the other hand $M_f^{\sing}$ is

$$\left(
                                      \begin{array}{ccc}
                                         & \mbox{Jacobian matrix of}\ q &  \\
                                        2x & 2y & 0 \\
                                      \end{array}
                                    \right)$$
and so $\delta_{i}^{\sing}=2xa_i(x,y,z)+2yb_i(x,y,z)$ for some functions $a_i,b_i$. So $\delta_i^{\sing}(0,0,z)=0$ for $i=1,2,3$. This means that $(0,0,z)$ is a solution in $V(M_f^{\sing})$ but since it does not intersect the cylinder $C_q$ this singularity lies at infinity.
\end{proof}

\section{The generic curvature locus of a $3$-manifold}

Finding necessary and sufficient conditions in terms of geometric invariants to characterize the topological type of the curvature locus
at a given point of a $3$-manifold seems to be a hard task.

The curvature locus being substantial or not is affine invariant but its topological type is not affine invariant for 3-manifolds. Our goal in this section
is to classify the generic curvature locus of each $\mathcal G$-orbit of a quadratic map $Q: \mathbb R^3 \to \mathbb R^3.$

More precisely, with the notation of  Section \ref{subsec:nets}, given a net
$Q=(q_1,q_2,q_3)\in \Gamma$, we can naturally associate the $2$-jet of parametrisation of a smooth $3$-manifold in the Monge form:
$$(x,y,z)\mapsto(x,y,z,q_1(x,y,z),q_2(x,y,z),q_3(x,y,z)),$$
whose second fundamental form at the origin is given by the quadratic map $Q$. The curvature locus of a normal form in Tables \ref{caseAB} and \ref{other types} are not generic in general from the geometrical point of view. Constructing $M_f$ with those normal forms and $\mathbb{S}^2=\rho^{-1}(0)$ where $\rho(x,y,z)=x^2+y^2+z^2-1$ may not give the best possible situation.

%For the matrix $M_f$ we consider the action of the group $\bar{\mathcal G}=\mathcal R\times\mathcal H$ where $\mathcal R$ is changes of coordinates in the source and $\mathcal H=GL(4)\times GL(3)$ acts naturally by multiplication to the left and to the right. So 
The classification of $\mathcal G$-orbits can be refined by the $\bar{\mathcal G}$-classification of the matrices corresponding to the nets in each $\mathcal G$-orbit.

Consider a generic positive definite homogeneous map $p:\mathbb{R}^3\rightarrow\mathbb{R}$,
\begin{equation}\label{p-homogeneous}
 p(x,y,z)=A_1x^2+A_2xy+A_3y^2+A_4xz+A_5yz+A_6z^2.   
\end{equation}

Our aim is to obtain a linear map $T:\mathbb{R}^3\rightarrow\mathbb{R}^3$ such that the following diagram commutes
$$
\xymatrix{
	\mathbb{R}^3 \ar[r]^{\rho} & \mathbb{R} \\
	\mathbb{R}^3 \ar[u]^{T} \ar[ur]_{p} & }.
$$

Changing $\rho$ for $p$ in the associated determinantal matrix we get
$$\tilde{M}_f=\left(
\begin{array}{ccc}
	\frac{\partial j^2f_1}{\partial x} & \frac{\partial j^2f_1}{\partial y} & \frac{\partial j^2f_1}{\partial z} \\
	\frac{\partial j^2f_2}{\partial x} & \frac{\partial j^2f_2}{\partial y} & \frac{\partial j^2f_2}{\partial z} \\
	\frac{\partial j^2f_3}{\partial x} & \frac{\partial j^2f_3}{\partial y} & \frac{\partial j^2f_3}{\partial z} \\
	\frac{\partial p}{\partial x} & \frac{\partial p}{\partial y} & \frac{\partial p}{\partial z} \\
\end{array}
\right).$$

\begin{prop}\label{generic}
	The generic normal form (in the sense of the $\bar{\mathcal G}$-classification) of the net $Q$ is given by $Q\circ T$ and the singularities of the generic curvature locus are given by $V(\tilde{M}_f)$.
\end{prop}

Now we can obtain the generic curvature locus in each $\mathcal G$-orbit.

\begin{teo}\label{genlocus} Let $Q \in \Gamma$ be a net of quadrics and
$$f(x,y,z)=(x,y,z, f_1(x,y,z), f_2(x,y,z), f_3(x,y,z))$$ a parametrization of $M_{\reg}^3\subset\mathbb{R}^6$ whose second fundamental form at the origin is $Q$.
%	Let $Q\in \Gamma$ be a net of quadrics and $M^3_{\reg}$ the $3$-manifold associated to $Q$. 
Then, for each possible $\mathcal{G}$-orbit	of $Q$, Table \ref{generical types} provides the generic curvature locus of $M^3_{\reg}$ at the origin.
	\begin{table}[h]%[tp]
		\caption{Generic curvature locus for regular manifolds}
		\centering
		{\begin{tabular}{cc}
				\hline
				Curvature locus & $\mathcal{G}$-Orbit \\
				\hline
				Roman Steiner &  $A,\ B,\ C,\ D_a,\ D_a^{*},\ D_c^{*},\ E_a,\ E_a^{*},\ F_a$ \\
				Cross-cap & $A,\ B_a,\ B_a^*,\ D_a,\ D_a^{*},\ D_b, \ D_b^{*},\ D_c,\ D_c^{*},\ E_b,\ E_b^{*},\ F_a,\ F_b$ \\
				Type $5$ & $F_a^{*},\ F_b^{*},\ G,\ G^*$\\
				Type $6$ & $H$\\
				Cone & $I$  \\
				Ellipsoid & $I^*$  \\
				\hline
			\end{tabular}
		}
		\label{generical types}
	\end{table}
\end{teo}
\begin{proof}
	Consider $Q\in \Gamma$ such that $q$ is in one of the first $3$-orbits:  $A$, $B$ or $C$. Thus,
	$$Q=(2xz+y^2,2yz,-x^2-2gy^2+cz^2+2gxz),\ c(c+9g^2)\neq0.$$
	To determine the generic topological type of the curvature locus, we must investigate the number of singular points
	of $Q$ when restricted to the unit sphere. In order to do so we calculate the $3\times3$-minors of the matrix
	$$\Lambda=
	\left(
	\begin{array}{ccc}
		& \mbox{Jacobian matrix of}\ Q &  \\
		2x & 2y & 2z \\
	\end{array}
	\right).
	$$
	Since we are only interested in the number of singular points, linear changes in $\Lambda$ do not interfere. Hence, after linear changes,
	we can rewrite $\Lambda$ as
	$$\Lambda_1=
	\left(
	\begin{array}{ccc}
		2z & 2y & 2x \\
		0 & 2z & 2y \\
		0 & 2y-6gy & 2z+2cz\\
		x & y & z \\
	\end{array}
	\right).
	$$
	The $3\times3$-minors of $\Lambda_1$ are
	$$
	\begin{array}{l}
		\delta=8z[(c+1)z^2+(3g-1)y^2];\\
		\delta_1=4(z^3-y^2z+xy^2-x^2z);\\
		\delta_2=4[(-3g-c)yz^2+(c+1)xyz+(3g-1)x^2y];\\
		\delta_3=4[(3g-1)x^2y+(c+1)xz^2].
	\end{array}
	$$
	Let $\delta=0$. Then $z=0$ or $(c+1)z^2+(3g-1)y^2=0$. If $z=0$, we have two solutions for $\delta=\delta_i=0$, $i=1,\ldots,3$.
	If $(c+1)z^2+(3g-1)y^2=0$, this equation has 0,1 or 2 solutions according to the sign of $\sigma=(c+1)(3g-1)$: positive, zero or negative, respectively.
	It is possible to show that if $(c+1)z^2+(3g-1)y^2=0$, we have four solutions of $\delta=\delta_i=0$, $i=1,\ldots,3$ whenever $c=-1$ and two solutions if $c\neq-1$ and $g=\frac{1}{3}$. Taking $\sigma=(c+1)(3g-1)<0$,
	$$
	\begin{array}{l}
		\delta_1=-4z\left(x^2+\frac{(c+1)}{(3g-1)}xz-\frac{(3g+c)}{(3g-1)}z^2\right)=-4z\xi(x,z)  \\
		\delta_2=\frac{4y}{(3g-1)}\xi(x,z) \\
		\delta_3=0.
	\end{array}
	$$
	The equation $\xi(x,z)=0$ may have 1 or 2 solutions, since its discriminant $\Delta_{\xi}=c^2+2c+1+36g^2+12g+12gc-4c\geq0$. This discriminant is represented by a pair of coincident lines in the plane $(g,c)$. Therefore, $\delta=\delta_i=0$, $i=1,\ldots,3$ have 4 solutions if $\Delta_{\xi}>0$ and 2 if $\Delta_{\xi}=0$.
	
	Each one of the orbits $A$, $B$ and $C$ intersects a region of the plane $(g,c)$ where $\sigma=(c+1)(3g-1)<0$ and $\Delta_{\xi}>0$, hence every orbit have generically six singular points and, therefore, the associated $3$-manifold has, at the origin, a Roman Steiner surface as its curvature locus.  Also, the orbits $A$, $B_a$ and $B_a^*$ intersect regions of the plane $(g,c)$ where $\sigma=(c+1)(3g-1)>0$ and $\Delta_{\xi}>0$, where two real and four complex solutions can be found. Hence in those orbits and the curvature locus can be a cross-cap surface as well. See Figure \ref{orbitas}.
	\begin{figure}[h!]
		\begin{center}
			\includegraphics[scale=0.4]{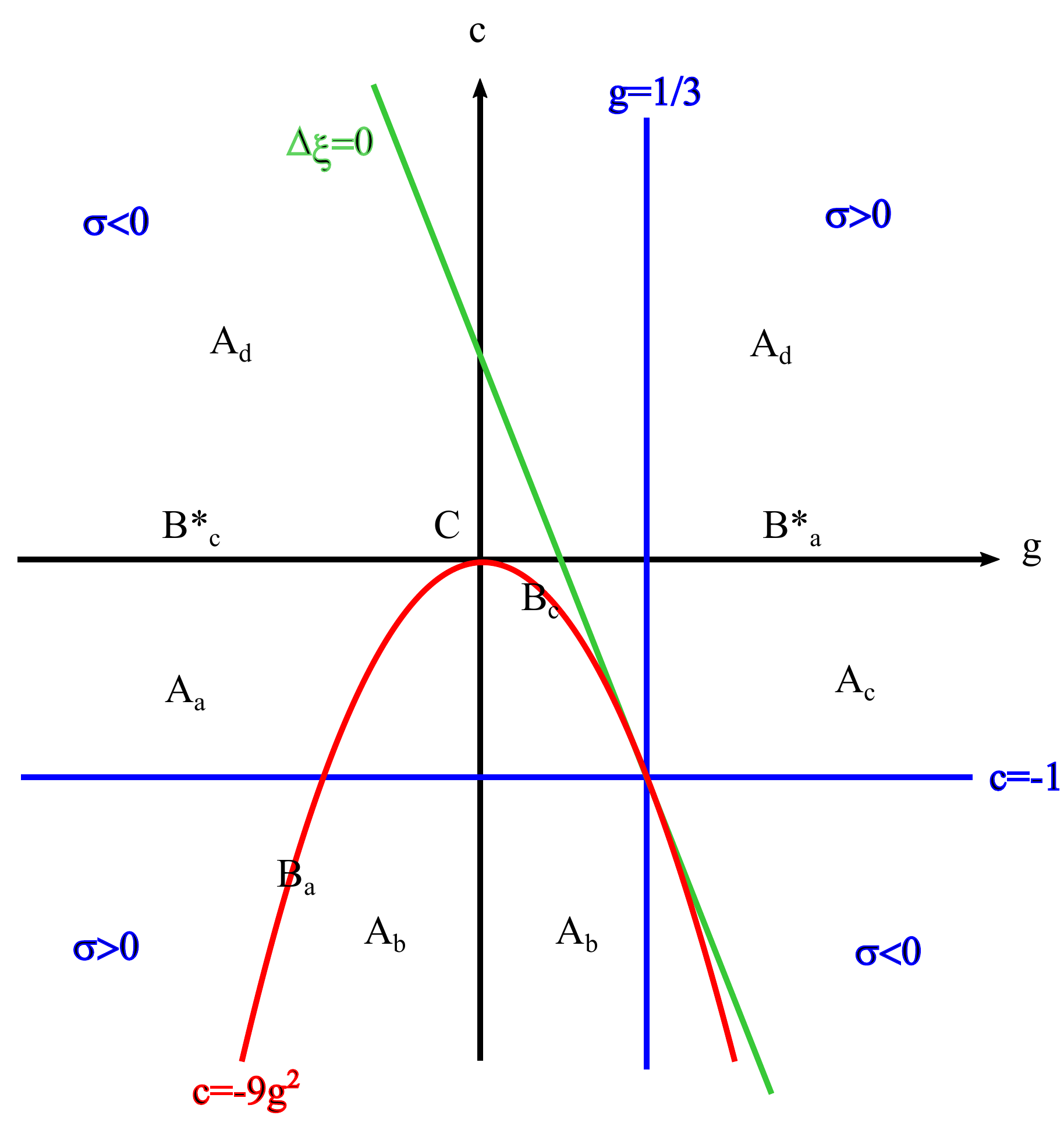}
			\caption{Bifurcation set and orbits.}
			\label{orbitas}
		\end{center}
	\end{figure}
	
	The remaining cases are similar. However, we need Proposition \ref{generic}. Consider for example the orbit $F_a$ in Table \ref{other types}, given by $q=(x^2+y^2,2xy,2yz)$. The matrix of $Q$ restricted to $p^{-1}(0)$ is
	$$\Lambda_{F_a}=\left(
	\begin{array}{ccc}
		& \mbox{Jacobian matrix of}\ q &  \\
		\frac{\partial p}{\partial x} & \frac{\partial p}{\partial y} & \frac{\partial p}{\partial z} \\
	\end{array}
	\right),$$
	where $p$ is as in (\ref{p-homogeneous}).
	Calculating the $3\times3$ minors $\delta_i$, $i=0,\ldots,3$, we obtain six solutions for $\delta_i=0$ when $A_4^2+4A_6(A_3-A_1)>0$, four solutions when $A_4^2+4A_6(A_3-A_1)=0$ and two solutions otherwise. Hence, we may obtain a Roman Steiner surface, Type 5 or a Cross-cap as the curvature locus of the associated $3$-manifold at the origin, respectively. The generic cases, however are the Roman Steiner and the cross-cap. In order to obtain a Roman Steiner surface as a generic curvature locus, for example, we may consider $A_1=A_4=2$, $A_2=A_5=0$, $A_3=4$ and $A_6=1$. Hence $p(x,y,z)=x^2+4y^2+(x+z)^2$, $T(X,Y,Z)=(X,\frac{Y}{2},Z-X)$ and the normal form of the net is given by $(Q\circ T)(X,Y,Z)=(X^2+\frac{Y^2}{4},XY,Y(Z-X))$.
\end{proof}

In fact, applying Proposition \ref{generic} to each normal form we get

\begin{coro}
	For each orbit in Table \ref{other types}, Table \ref{normal forms1} shows a normal form of the net for which we obtain the generic curvature locus of the associated $3$-manifold.
	\begin{table}[h]%[tp]
		\caption{Normal forms}
		\centering
		{\begin{tabular}{lc}
				\hline
				Name & Normal form \\
				\hline
				$D_a$ (Roman Steiner) &  $(X^2,Y^2,(Z+2Y)^2+2XY)$ \\ 
				$D_a$ (Cross-cap) & $((X+\frac{\sqrt{2}}{2}Z)^2,(\frac{Y}{2}-\frac{\sqrt{2}}{4}Z)^2,2Z^2+(X+\frac{\sqrt{2}}{2}Z)(Y-\frac{\sqrt{2}}{2}Z))$  \\
				$D_b,\ D_c$ & $(X^2-Y^2,2XY,X^2\pm(Z+2Y)^2)$  \\
				$D_a^*$ (Roman Steiner) & $(XZ,YZ,\frac{Z^2}{4}+2XY)$ \\
				$D_a^*$ (Cross-cap) & $(XZ,2YZ,Z^2+XY)$ \\
				$D_b^*$ & $(XZ,\frac{YZ}{3},X^2+\frac{Y^2}{9}- \frac{Z^2}{4})$ \\ 
				$D_c^*$ (Steiner Romana) & $(XZ,\frac{YZ}{3},X^2+\frac{Y^2}{9}+ \frac{Z^2}{4})$ \\
				$D_c^*$ (Cross-cap) & $(\frac{XZ}{3},\frac{2YZ}{3},\frac{X^2}{4}+Y^2+\frac{Z^2}{9})$ \\
				$E_a,\ E_b$ & $(X^2\mp Y^2,2XY,(Z-Y)^2)$  \\
				$E_a^*,\ E_b^*$ & $(X^2-Y^2,2XZ,2YZ)$, $(X^2+\frac{Y^2}{4},2XZ,YZ)$    \\
				$F_a\ \mbox{(Steiner Romana)}$ & $(X^2+ \frac{Y^2}{4},XY,Y(Z-X))$\\ 
				$F_a$ (Cross-cap) & $(\frac{X^2}{4}+Y^2,XY,2YZ)$ \\
				$F_b$ & $(X^2- \frac{Y^2}{4},XY,Y(Z-X))$ \\
				$F_a^*,\ F_b^*$ & $((X-Y)^2\mp Y^2,2(X-Y)Z,Z^2)$  \\
				$G$ &  $(X^2,(Y-X)^2,2(Y-X)(Z-X))$  \\
				$G^*$ & $(2XY,2XZ,Z^2)$  \\
				$H$ & $(X^2,2XY,Y^2+2XZ)$ \\
				$I$ & $(X^2,2XY,Y^2)$ \\
				$I^*$ & $(2XZ,2YZ,Z^2)$ \\
				\hline
			\end{tabular}
		}
		\label{normal forms1}
	\end{table}

%%%%%%%%%%%%%%%%%%%%%%%%%%%%%%%%%%%  Tabela Antiga (está em comentário!!!!) %%%%%%%%%%%%%%%%%%%%%%%%%%%%%%%%%%%%%%%%%%%%%%%%%%%%%%%%%%%%%%%%%%%
\begin{comment}
	\begin{table}[h]%[tp]
		\caption{Normal forms}
		\centering
		{\begin{tabular}{cc}
				\hline
				Name & Normal form \\
				\hline
				$D_a$ &  $(X^2,Y^2,(Z+2Y)^2+2XY)$, $((X+\frac{\sqrt{2}}{2}Z)^2,(\frac{Y}{2}-\frac{\sqrt{2}}{4}Z)^2,2Z^2+(X+\frac{\sqrt{2}}{2}Z)(Y-\frac{\sqrt{2}}{2}Z))^1$  \\
				$D_b,\ D_c$ & $(X^2-Y^2,2XY,X^2\pm(Z+2Y)^2)$  \\
				$D_a^*$ & $(XZ,YZ,\frac{Z^2}{4}+2XY)$, $(XZ,2YZ,Z^2+XY)^1$ \\
				$D_b^*,\ D_c^*$ & $(XZ,\frac{YZ}{3},X^2+\frac{Y^2}{9}\mp \frac{Z^2}{4})$, $(\frac{XZ}{3},\frac{2YZ}{3},\frac{X^2}{4}+Y^2+\frac{Z^2}{9})^1$ \\
				$E_a,\ E_b$ & $(X^2\mp Y^2,2XY,(Z-Y)^2)$  \\
				$E_a^*,\ E_b^*$ & $(X^2-Y^2,2XZ,2YZ)$, $(X^2+\frac{Y^2}{4},2XZ,YZ)$    \\
				$F_a,\ F_b$ & $(X^2\pm \frac{Y^2}{4},XY,Y(Z-X))$, $(\frac{X^2}{4}+Y^2,XY,2YZ)^1$ \\
				$F_a^*,\ F_b^*$ & $((X-Y)^2\mp Y^2,2(X-Y)Z,Z^2)$  \\
				$G$ &  $(X^2,(Y-X)^2,2(Y-X)(Z-X))$  \\
				$G^*$ & $(2XY,2XZ,Z^2)$  \\
				$H$ & $(X^2,2XY,Y^2+2XZ)$ \\
				$I$ & $(X^2,2XY,Y^2)$ \\
				$I^*$ & $(2XZ,2YZ,Z^2)$ \\
				\hline
			\end{tabular}
		}
		\label{normal forms}
	\end{table}
\end{comment}
%%%%%%%%%%%%%%%%%%%%%%%%%%%%%%%%%%%%%%%%%%%%%%%%%%%%%%%%%%%%%%%%%%%%%%%%%%%%%%%%%%%%%%%%%%%%%%%%%%%%%%%%%%%%%%%%%%%%%%%%%
\end{coro}

%%%%%%%%%%%%%%%%%%%%%%%%%%%%%%%%%%%%%%%%%%%%%%%%%%%%%%%%%%%%% %%%%%%%%%%%%%%%%%%%%%%%%%%%%%%%%%%%%%%%%%%%%%%%%%%%%%
We can also deduce the generic curvature locus for the same $\mathcal G$-orbits but for singular manifolds.

\begin{coro} Let $Q \in \Gamma$ be a net of quadrics and
$$f(x,y,z)=(x,y, f_1(x,y,z), f_2(x,y,z), f_3(x,y,z))$$ a parametrization of $M_{\sing}^3\subset\mathbb{R}^5$ whose second fundamental form at the origin is $Q$.
%Let $Q\in \Gamma$ be a net of quadrics and $M^3_{\sing}\subset\mathbb{R}^5$ the $3$-manifold associated to $Q$. 
Then, for each possible $\mathcal{G}$-orbit of $Q$, Table \ref{generical types singular} provides the generic curvature locus of $M^3_{\sing}$ at the origin.
	\begin{table}[h]%[tp]
		\caption{Generic curvature locus for singular manifolds}
		\centering
		{\begin{tabular}{cc}
				\hline
				Curvature locus & $\mathcal{G}$-Orbit \\
				\hline
				6 CC &  $A,\ B,\ C,\ D_a,\ D_a^{*},\ D_c^{*},\ E_a,\ E_a^{*},\ F_a$ \\
				2 CC & $A,\ B_a,\ B_a^*,\ D_a,\ D_a^{*},\ D_b, \ D_b^{*},\ D_c,\ D_c^{*},\ E_b,\ E_b^{*},\ F_a,\ F_b$ \\
				2 CC and 2 TCC & $F_a^{*},\ F_b^{*},\ G,\ G^*$\\
				2 DTCC & $H$\\
				Ellipse & $I$  \\
				Paraboloid & $I^*$  \\
				\hline
			\end{tabular}
		}
		\label{generical types singular}
	\end{table}
\end{coro}
\begin{proof}
For all $\mathcal G$-orbits except for $I$, there exists a non-asymptotic direction in $M^3_{\reg}$. In this case, if we take the normal form which yields the generic curvature locus as in Table \ref{normal forms1} and project in a non-asymptotic direction, by Theorem \ref{iso}, the curvature locus of the singular projection will have the same number and type of singularities as the curvature locus in the regular case. In most cases (all orbits except for $D_a^{*}$, $D_b^{*}$, $D_c^{*}$, $I^*$) $(0,0,1)$ is an asymptotic direction, so we must chose a different direction to do the projection. This projection will force a linear change of coordinates in the tangent space which will eventually force a linear change of coordinates in the source in the normal forms of Table \ref{normal forms1}. However, this change of coordinates lies in $\mathcal G$ and so we do not change the type of $\mathcal G$-orbit. For all, except for the above mentioned 4 orbits, the generic normal form will change.

For the orbit $I$, all tangent directions are asymptotic so the number and type of singularities of a generic curvature locus in this orbit for the singular case will change. Notice that any singular 3-manifold with the associated net in this $\mathcal G$-orbit will have a parametrisation $\mathscr A^2$-equivalent to $(x,y,0,0,0)$. This means that the 2-jet depends only on $x$ and $y$ and so the generic curvature locus in this orbit is an ellipse.
\end{proof}

\end{document}